\documentclass[a4paper,12pt]{amsart}
\usepackage{amsmath}
\usepackage{amsmath}
\usepackage{amssymb}
\usepackage{amscd}
\usepackage{mathrsfs}
\usepackage[all]{xy}
\usepackage[dvipdfm]{graphicx}
\def\mathcal{\mathscr}
\everymath{\displaystyle}
\newtheorem{thm}{Theorem}[section]
\newtheorem{lem}[thm]{Lemma}
\newtheorem{cor}[thm]{Corollary}

\theoremstyle{definition}

\newtheorem{rem}[thm]{Remark}

\newtheorem{defn}[thm]{Definition}

\newcommand{\mca}[1]{{\mathcal{#1}}}

\def\Z{{\mathbb Z}}
\def\C{{\mathbb C}}
\def\R{{\mathbb R}}

\def\image{\text{\rm Im}\,}
\def\Diff{\text{\rm Diff}\,}

\def\dR{\text{\rm dR}}
\def\ep{\varepsilon} 

\def\ech{\text{\rm ECH}}
\def\fix{\text{\rm Fix}}

\def\Ham{\text{\rm Ham}\,}

\def\id{\text{\rm id}}
\def\interior{\text{\rm int}\,}

\def\PD{\text{\rm PD}}
\def\ph{\varphi} 

\def\supp{\text{\rm supp}\,}
\def\vol{\text{\rm vol}}

\begin{document}
\pagestyle{plain}
\thispagestyle{plain}

\title[A $C^\infty$ closing lemma for Hamiltonian diffeomorphisms of closed surfaces]
{A $C^\infty$ closing lemma for Hamiltonian diffeomorphisms of closed surfaces}

\author[Masayuki Asaoka]{Masayuki Asaoka}
\address{Department of Mathematics, 
Kyoto University,
Kyoto 606-8502, Japan}
\email{asaoka@math.kyoto-u.ac.jp}

\author[Kei Irie]{Kei Irie}
\address{Research Institute for Mathematical Sciences, Kyoto University,
Kyoto 606-8502, Japan
/
Simons Center for Geometry and Physics, State University of New York, 
Stony Brook, NY 11794-3636, USA (on visit)} 
\email{iriek@kurims.kyoto-u.ac.jp}


\begin{abstract}
We prove a $C^\infty$ closing lemma for Hamiltonian diffeomorphisms of closed surfaces. 
This is a consequence of a $C^\infty$ closing lemma for Reeb flows on closed contact three-manifolds, 
which was recently proved as an application of spectral invariants in embedded contact homology. 
A key new ingredient of this paper is an analysis of an area-preserving map near its fixed point, 
which is based on some classical results in Hamiltonian dynamics:
existence of KAM invariant circles for elliptic fixed points, and 
convergence of the Birkhoff normal form for hyperbolic fixed points. 
\end{abstract}

\maketitle

\section{Introduction} 

The aim of this paper is to prove a $C^\infty$ closing lemma for 
Hamiltonian diffeomorphisms of closed surfaces. 
Let us first introduce some notations. 
For any closed surface (i.e., $C^\infty$ two-manifold) $S$, 
let $\Diff(S)$ denote the group of all $C^\infty$ diffeomorphisms of $S$, equipped with the $C^\infty$ topology. 
For any $\ph \in \Diff(S)$, let $\fix(\ph)$ denote the set of fixed points of $\ph$, 
and $\mca{P}(\ph)$ denote the set of periodic points of $\ph$: 
\[ 
\fix(\ph):= \{ x \in S \mid \ph(x)= x\}, \qquad
\mca{P}(\ph):= \bigcup_{m=1}^\infty \fix(\ph^m).
\]
Also, the closure of $\{ x \in S \mid \ph(x) \ne x\}$ is called the support of $\ph$, and denoted as $\supp \ph$. 

When $S$ is equipped with an area form (i.e., nowhere vanishing $2$-form) $\omega$, let 
\[
\Diff(S,\omega):= \{ \ph  \in \Diff(S) \mid \ph^*\omega = \omega\}, 
\]
which is the group of area-preserving diffeomorphisms. 
For any $h \in C^\infty(S)$, we define its Hamiltonian vector field $X_h$ by $i_{X_h}\omega = -dh$. 
Our convention for the interior product $i$ is $i_{X_h} \omega (\, \cdot \,) = \omega(X_h, \, \cdot \,)$. 
For any $H \in C^\infty([0,1] \times S)$ and $t \in [0,1]$, we define $H_t \in C^\infty(S)$ by $H_t(x):= H(t,x)$, 
and $(\ph^t_H)_{t \in [0,1]}$ denotes the isotopy on $S$ defined by $\ph^0_H = \id_S$ and $\partial_t \ph^t_ H= X_{H_t}(\ph^t_H)$. 
Then we define 
\[
\Ham(S, \omega):= \{ \ph^1_H \mid  H \in C^\infty([0,1] \times S)\},
\]
which is the group of Hamiltonian diffeomorphisms. 
It is known that $\Ham(S,\omega)=\Diff(S,\omega)$ when $S$ is homeomorphic to the two-sphere. 

Throughout this paper, $\Ham(S, \omega)$ and $\Diff (S, \omega)$ are equipped with topologies induced 
from the $C^\infty$ topology on $\Diff(S)$. 
Now we can state our main result as follows:

\begin{thm}\label{151207_1}
Let $S$ be any closed, oriented surface, 
$\omega$ be any area form on $S$, and $\ph \in \Ham(S, \omega)$. 
For any nonempty open set $U \subset S$, 
there exists a sequence $(\ph_j)_{j \ge 1} $in $\Ham(S,\omega)$ such that 
$\mca{P}(\ph_j) \cap U \ne \emptyset$ for every $j \ge 1$ and $\lim_{j \to \infty} \ph_j= \ph$. 
\end{thm} 

Using standard arguments, 
we can prove a $C^\infty$ general density theorem for Hamiltonian diffeomorphisms.

\begin{cor}\label{160908_1} 
For any $(S, \omega)$ as in Theorem \ref{151207_1}, 
\[
\{ \ph \in \Ham (S, \omega) \mid \text{$\mca{P}(\ph)$ is dense in $S$} \, \}
\]
is residual (i.e., contains a countable intersection of open and dense sets) in $\Ham(S, \omega)$. 
\end{cor} 
\begin{proof}
For any nonempty open set $U \subset S$, 
let $\mca{H}_U$ denote the set consisting of $\ph \in \Ham(S,\omega)$ 
such that there exists a nondegenerate periodic orbit of $\ph$ which intersects $U$. 
Then, $\mca{H}_U$ is obviously open in $\Ham(S,\omega)$, 
and dense by Theorem \ref{151207_1}. 
Let  $(U_i)_{i \in  I}$ be any countable basis of open sets in $S$. 
Then $\bigcap_{i \in I} \mca{H}_{U_i}$  is residual in $\Ham(S,\omega)$, and 
if $\ph \in \bigcap_{i \in I} \mca{H}_{U_i}$ then $\mca{P}(\ph)$ is dense in $S$. 
\end{proof}

We can also prove a $C^r$ general density theorem ($1 \le r \le \infty$) for area-preserving diffeomorphisms of the two-sphere. 

\begin{cor}\label{160908_2} 
Let $r$ be a positive integer or $\infty$, and let $\Diff_r(S, \omega)$ denote the set of $C^r$ diffeomorphisms of $S$ preserving $\omega$. 
When $S$ is homeomorphic to the two-sphere, 
\[ 
\{ \ph \in \Diff_r(S, \omega) \mid \text{ $\mca{P}(\ph)$ is dense in $S$}\,\}
\]
is residual in $\Diff_r(S, \omega)$ with the $C^r$ topology.
\end{cor} 
\begin{proof} 
The case $r=\infty$ is immediate from Corollary \ref{160908_1}, since $\Diff_\infty(S, \omega) = \Ham (S, \omega)$. 
The case $1 \le r< \infty$ follows from the case $r=\infty$ and the fact that $\Diff_\infty(S, \omega)$ is dense in $\Diff_r(S,\omega)$ with the $C^r$ topology, 
which is proved in \cite{Zehnder}. 
\end{proof} 

\begin{rem}[Historical remarks]
The $C^1$ closing lemma (and general density theorem) was first proved for nonconservative dynamics by Pugh \cite{Pugh1} \cite{Pugh2}, 
and later proved for conservative dynamics by Pugh-Robinson \cite{PR}. 
In particular, \cite{PR} established the $C^1$ closing lemma for symplectic and volume-preserving diffeomorphisms in arbitrary dimensions. 
On the other hand, a $C^r$ closing lemma for $r \ge 2$ is not established except for a few cases (see \cite{Anosov_Zhuzhoma} Section 5), 
and it has been considered as an important open problem in the theory of dynamical systems (in particular, see Smale \cite{Smale} Problem 10). 

As for the $C^r$ general density theorem for area-preserving diffeomorphisms of closed surfaces (which is completely settled for the two-sphere by Corollary \ref{160908_2}), 
as far as the authors know 
the only affirmative result so far is that of Xia \cite{Xia} 
(which generalizes the previous result by Franks - Le Calvez \cite{FLC} for the two-sphere)
which states that for a $C^r$ generic area-preserving diffeomorphism of a closed surface (where $1 \le r \le \infty$), 
the union of (un)stable manifolds of its hyperbolic periodic points is dense in the surface. 
\end{rem} 

Theorem \ref{151207_1} is a consequence of 
a $C^\infty$ closing lemma for Reeb flows on closed contact three-manifolds (Lemma \ref{151030_1}), 
which was proved in \cite{Irie}. 
For the convenience of the reader, 
we sketch its proof in Section 2. 
The proof uses recent developments in quantitative aspects of embedded contact homology, in particular the result 
in \cite{CGHR} by 
Cristofaro-Gardiner, Hutchings and Ramos. 

In Section 3, we prove 
a $C^\infty$ closing lemma for area-preserving diffeomorphisms of a surface with boundary, 
which satisfy some technical conditions (Lemma \ref{151207_2}). 
The idea of the proof is to regard an area-preserving map as a ``return map'' of a certain Reeb flow, 
which is inspired by a recent paper \cite{Hutchings_Calabi} by Hutchings. 

In Section 4, we prove Theorem \ref{151207_1} using Lemma \ref{151207_2} and an analysis of an area-preserving map 
near its fixed point. 
We exploit some classical results in Hamiltonian dynamics: 
existence of KAM invariant circles for elliptic fixed points, and 
convergence of the Birkhoff normal form for hyperbolic fixed points. 

\section{Reeb flows on contact three-manifolds} 

Let $(M, \lambda)$ be a contact manifold, where $\lambda$ denotes the contact form, 
with the contact distribution $\xi_\lambda:= \ker \lambda$. 
The Reeb vector field $R_\lambda$ is defined by equations $\lambda(R_\lambda) = 1$, $d\lambda(R_\lambda, \cdot \, ) = 0$. 
Let $\mca{P}(M, \lambda)$ denote the set of periodic orbits of $R_\lambda$, namely 
\[
\mca{P}(M, \lambda):= \{ \gamma: \R/T_\gamma \Z \to M  \mid T_\gamma >0, \, \dot{\gamma} = R_\lambda(\gamma)\}.
\]
Lemma \ref{151030_1} below is proved as a claim in the proof of \cite{Irie} Lemma 3.1 (in a slightly weaker form). 
The aim of this section is to sketch its proof, referring to \cite{Irie} for details. 

\begin{lem}[\cite{Irie}]\label{151030_1}
Let $(M, \lambda)$ be a closed contact three-manifold. 
For any $h \in C^\infty(M, \R_{\ge 0}) \setminus \{0\}$, 
there exist $t \in [0,1]$ and 
$\gamma \in \mca{P}(M, (1+th)\lambda)$ which intersects $\supp h$. 
\end{lem} 

Our proof of Lemma \ref{151030_1} is based on embedded contact homology (ECH). 
For any closed contact three-manifold $(M, \lambda)$ and 
$\Gamma \in H_1(M: \Z)$, 
this theory assigns
a $\Z/2\,$-vector space\footnote{ECH can be defined with $\Z$ coefficients, however $\Z/2$ coefficients are sufficient for our purpose.}
$\ech(M, \xi_\lambda, \Gamma)$, which is relatively $\Z\,$-graded 
if $c_1(\xi_\lambda) + 2 \PD(\Gamma) \in H^2(M: \Z)$ is torsion
($c_1$ denotes the first Chern class, and $\PD$ denotes the Poincar\'{e} dual). 
It is easy to see that such $\Gamma$ exists for any $(M,\xi_\lambda)$. 

For each $\sigma \in \ech(M, \xi_\lambda, \Gamma) \setminus \{0\}$, 
one can assign 
$c_\sigma(M, \lambda) \in \R_{\ge 0}$ (see \cite{Hutchings_Quantitative} Section 4.1), 
the associated \textit{ECH spectral invariant} 
(this term is not used in \cite{Hutchings_Quantitative}). 
One can prove the following properties: 
\begin{enumerate} 
\item[(a):] $c_\sigma(M, \lambda) \in \{0\} \cup \biggl\{ \sum_{i=1}^m T_{\gamma_i} \bigg{|} m \ge 1, \, \gamma_1, \ldots, \gamma_m \in \mca{P}(M, \lambda) \biggr\}.$ 
\item[(b):] If a sequence $(f_j)_{j \ge 1}$ in $C^\infty(M, \R_{>0})$ satisfies 
$\lim_{j \to \infty} \| f_j - 1 \|_{C^0} = 0$, then $\lim_{j \to \infty} c_\sigma(M, f_j \lambda) = c_\sigma(M, \lambda)$. 
\item [(c):] 
Let $\Gamma \in H_1(M: \Z)$ be such that $c_1(\xi_\lambda) + 2\PD(\Gamma)$ is torsion, and let $I$ denote an arbitrary absolute $\Z\,$-grading on 
$\ech(M, \xi_\lambda, \Gamma)$. 
\begin{enumerate}
\item[(i):] 
$\ech(M, \xi_\lambda, \Gamma)$ is unbounded from above with this $\Z\,$-grading. 
\item[(ii):] 
If $M$ is connected, 
for any sequence $(\sigma_k)_{k \ge 1}$ of nonzero homogeneous elements in $\ech(M, \xi_\lambda, \Gamma)$ 
satisfying $\lim_{k \to \infty} I(\sigma_k) = \infty$, there holds 
\[ 
\lim_{k \to \infty} \frac{ c_{\sigma_k}(M, \lambda)^2}{I(\sigma_k)} = \int_M \lambda \wedge d \lambda =: \vol(M, \lambda).
\]
\item[(iii):]
For any $h \in C^\infty(M, \R_{\ge 0}) \setminus \{0\}$, 
there exists $\sigma \in \ech(M, \xi_\lambda, \Gamma)$ such that 
$c_\sigma(M, (1+h)\lambda)   > c_\sigma(M, \lambda)$. 
\end{enumerate}
\end{enumerate} 

(a) is \cite{Irie} Lemma 2.4 (a special case is proved in \cite{CGH} Lemma 3.1). 
(b) is explained in \cite{CGH} Section 2.5 as the ``Continuity axiom''. 
(c)-(i) follows from Seiberg-Witten Floer theory (see \cite{CGH} Section 2.6). 
(c)-(ii) is \cite{CGHR} Theorem 1.3. 
(c)-(iii) follows from (i) and (ii), since $\vol(M, (1+h)\lambda) > \vol(M, \lambda)$. 

We also need 
Lemma \ref{151117_1} below, 
which is proved by elementary arguments using Sard's theorem
(see \cite{Irie} Section 2.1 for details). 

\begin{lem}[\cite{Irie} Lemma 2.2]\label{151117_1}
For any closed contact manifold $(M, \lambda)$, 
\[
\mca{A}(M, \lambda)_+:=  \{0\} \cup \biggl\{ \sum_{i=1}^m T_{\gamma_i} \bigg{|} m \ge 1, \, \gamma_1, \ldots, \gamma_m \in \mca{P}(M, \lambda) \biggr\}
\]
is a null (i.e., Lebesgue measure zero) set. 
\end{lem}

\begin{proof}[\textbf{Proof of Lemma \ref{151030_1}}]
We may assume that $M$ is connected, 
and we set $\lambda_t:= (1+th)\lambda$ for any $t \in [0,1]$. 
Suppose that the lemma does not hold, i.e., 
$\gamma \in \mca{P}(M, \lambda_t) \implies \image \gamma \cap \supp h= \emptyset$ for any $t \in [0,1]$. 
Then $\mca{P}(M, \lambda_t) = \mca{P}(M, \lambda)$ for any $t \in [0,1]$, 
since $R_{\lambda_t} = R_\lambda$ on $M \setminus \supp h$. 
Hence $\mca{A}(M, \lambda_t)_+ = \mca{A}(M, \lambda)_+$ for any $t \in [0,1]$. 

For any $\Gamma \in H_1(M: \Z)$, 
$\sigma \in \ech(M, \xi_\lambda, \Gamma)\setminus \{0\}$ and $t \in [0,1]$,  (a) shows that 
\[ 
c_\sigma(M, \lambda_t) \in \mca{A}(M, \lambda_t)_+ = \mca{A}(M, \lambda)_+. 
\] 
(b) shows that $c_\sigma(M, \lambda_t)$ is continuous on $t \in [0,1]$. 
On the other hand, $\mca{A}(M, \lambda)_+$ is a null set (Lemma \ref{151117_1}). 
Thus $c_\sigma(M, \lambda_t)$ is constant on $t \in [0,1]$, in particular we obtain 
$c_\sigma(M, (1+h)\lambda) = c_\sigma(M, \lambda)$ for any $\sigma$, which contradicts (c)-(iii). 
\end{proof} 

\section{Return maps of Reeb flows} 

The aim of this section is to prove Lemma \ref{151207_2} below. 
$\Omega^1$ denotes the set of $C^\infty$ $1$-forms. 

\begin{lem}\label{151207_2} 
Let $S$ be any compact, connected surface with boundary such that $\partial S$ is diffeomorphic to $S^1$. 
Let $\omega$ be any area form on $S$ and 
$\ph \in \Diff(S, \omega)$ such that: 
\begin{itemize}
\item $\ph \equiv \id_S$ near $\partial S$.
\item For any $\beta \in \Omega^1(S)$ such that $d\beta = \omega$, $\ph^*\beta - \beta$ is exact. 
\end{itemize} 
Then, for any nonempty open set $U$ in $\interior S:= S \setminus \partial S$, 
there exists a sequence $(\ph_j)_{j \ge 1} $in $\Diff(S,\omega)$ such that 
$\lim_{j \to \infty} \ph_j = \ph$, and for every $j \ge 1$ there holds
\[
\mca{P}(\ph_j) \cap U  \ne \emptyset, \qquad
\supp (\ph^{-1} \circ \ph_j) \subset U. 
\]
\end{lem} 

Our idea to prove Lemma \ref{151207_2} is to realize $\ph|_{\interior S}$ as a ``return map'' of a certain Reeb flow. 
First we recall the following notion from \cite{HWZ_3D}. 

\begin{defn}\label{151112_2} 
Let $(M, \lambda)$ be a closed contact three-manifold, 
and let $(\ph^t)_{t \in \R}$ denote the flow on $M$ generated by $R_\lambda$, i.e., 
$\ph^0 = \id_M$ and $\partial_t \ph^t = R_\lambda(\ph^t)$. 
A \textit{global surface of section} in $(M, \lambda)$ is a 
compact surface $\Sigma$ with boundary, 
which is embedded in $M$ and satisfies the following conditions: 
\begin{itemize} 
\item Each connected component of $\partial \Sigma$ is a (image of) periodic orbit of $R_\lambda$. 
\item  $\interior \Sigma$ is transversal to $R_\lambda$.
\item For any $p  \in M \setminus \Sigma$, there exist $t_-(p) \in \R_{<0}$, $t_+(p) \in \R_{>0}$ such that 
$\ph^{t_-(p)}(p), \ph^{t_+(p)}(p) \in \Sigma$ and 
$t \in (t_-(p), t_+(p)) \implies \ph^t(p) \notin \Sigma$. 
\end{itemize} 
Let us define $\pi_\pm: M \setminus \Sigma \to \interior \Sigma$ by 
$\pi_\pm(p):= \ph^{t_\pm(p)}(p)$. 
We also define the \textit{return map}
$\rho_{M, \lambda, \Sigma}: \interior \Sigma \to \interior \Sigma$ so that
$\rho_{M, \lambda, \Sigma}(\pi_-(p)) = \pi_+(p)$ for any $p \in M \setminus \Sigma$. 
\end{defn}

It is easy to see that $\rho_{M, \lambda, \Sigma}$ preserves $d\lambda|_{\interior \Sigma}$. 
We abbreviate $\rho_{M,\lambda,\Sigma}$ as $\rho_\lambda$ when there is no risk of confusion. 

The next lemma is a small variation of \cite{Hutchings_Calabi} Proposition 2.1. 

\begin{lem}\label{151112_1} 
For any $(S, \omega, \ph)$ which satisfies the assumptions in Lemma \ref{151207_2}, 
there exists $(M, \lambda, \Sigma)$ such that 
$(\interior S, \omega, \ph|_{\interior S})$ is $C^\infty$ conjugate to 
$(\interior \Sigma, d\lambda, \rho_{M, \lambda, \Sigma})$. 
\end{lem} 
\begin{proof}
Let us take a Liouville vector field $V$ on $(S, \omega)$, i.e., 
$d(i_V \omega)=\omega$ and $V$ is outer normal to $\partial S$. 
We set $\beta:= i_V \omega$. 
There exists a local chart $(r,\theta)\, (\sqrt{1-\ep^2} \le r \le 1, \, \theta \in \R/\Z)$ 
near $\partial S$ such that 
$\partial S = \{r=1\}$ and $\beta = ar^2 d\theta$, 
where $a:= \int_S \omega$. 

Let $Y:= [0,1] \times S/ \sim$, 
where $\sim$ is defined as $(1,x) \sim (0, \ph(x))\, (x \in S)$. 
For any $t \in [0,1]$, we define an embedding $e_t: S \to Y$ by 
$e_t(x):= [(t,x)]$. 
Then, there exists a contact form $\lambda_Y$ on $Y$ such that: 
\begin{itemize} 
\item There exists $h \in \Z_{>0}$ such that $\lambda_Y= a (h dt + r^2 d\theta)$ near $\partial Y = [0,1] \times \partial S/\sim$. 
\item The Reeb vector field $R_{\lambda_Y}$ is parallel to $\partial_t$. 
\item $e_t^* d\lambda_Y = \omega$ for any $t \in [0,1]$. 
\end{itemize} 
$\lambda_Y$ is defined as follows: 
since
$\ph^* \beta - \beta$ is exact, there exists $f \in C^\infty(S)$ such that $\ph^*\beta - \beta = df$. 
$f$ is constant near $\partial S$, since $\ph \equiv \id_S$ near $\partial S$. 
By adding a constant, we may assume that $\min f>0$ and 
$f \equiv ah$ near $\partial S$ for some $h \in \Z_{>0}$. 
Now we can proceed in exactly the same way as the proof of \cite{Hutchings_Calabi} Proposition 2.1. 

Let $Z=\R/\Z \times \{z \in \C \mid |z|<\ep\}$, and $C:= \R/\Z \times \{0\} \subset Z$. 
We define $M:= \interior Y \sqcup Z/\sim$, where $\sim$ is defined as
\[
(t, r, \theta) \sim (\tau, z=\rho e^{\sqrt{-1}\psi}) \iff r^2 + \rho^2 = 1, \, \psi = 2\pi t, \, \theta = \tau - ht. 
\]
Then, it is easy to see that $\lambda_Y|_{\interior Y}$ extends to a $C^\infty$ contact form $\lambda$ on $M$, such that 
$C$ is a periodic orbit of $R_\lambda$. 
Finally, $\Sigma:= \{0\} \times \interior S \cup C$ 
is a global surface of section in $(M, \lambda)$, and 
$(\interior \Sigma, d\lambda, \rho_{M, \lambda, \Sigma})$ is conjugate to 
$(\interior S, \omega, \ph|_{\interior S})$ via $e_0|_{\interior S}$. 
\end{proof}

\begin{proof}[\textbf{Proof of Lemma \ref{151207_2}}]
By Lemma \ref{151112_1}, 
there exists $(M, \lambda, \Sigma)$ such that 
$(\interior \Sigma, d\lambda, \rho_{M, \lambda, \Sigma})$ is conjugate to 
$(\interior S, \omega, \ph|_{\interior S})$ via a diffeomorphism 
$F: \interior S \to \interior \Sigma$. 

Let us take $h \in C^\infty(M, \R_{\ge 0}) \setminus \{0\}$ such that 
$\supp h \subset \pi_-^{-1}(F(U))$. 
By Lemma \ref{151030_1}, 
there exists a sequence $(t_j)_{j \ge 1}$ in $\R_{>0}$ such that $\lim_{j \to \infty} t_j = 0$ and 
for every $j$ there exists $\gamma_j \in \mca{P}(M,  (1+t_jh)\lambda)$ which intersects $\supp h$. 

Let $\lambda_j:= (1+t_jh) \lambda$. 
Then $\Sigma$ is a global surface of section in $(M, \lambda_j)$ for any sufficiently large $j$, 
and $\lim_{j \to \infty} \rho_{\lambda_j} = \rho_\lambda$ in the $C^\infty$ topology. 
Let $K:= \pi_-(\supp h) \subset F(U)$. 
For any such $j$, there holds 
$\supp (\rho_\lambda^{-1} \circ \rho_{\lambda_j}) \subset K$ and 
$\mca{P}(\rho_{\lambda_j}) \cap K \ne \emptyset$. 
Also, $\rho_{\lambda_j}$ preserves $d\lambda_j|_{\interior \Sigma} = d\lambda|_{\interior \Sigma}$. 
Moreover, $F^{-1} \circ \rho_{\lambda_j} \circ F \in \Diff(\interior S, \omega)$ 
extends to $\ph_j \in \Diff(S, \omega)$ by 
setting $\ph_j|_{\partial S}:= \id_{\partial S}$, 
since $\supp (\rho_{\lambda}^{-1} \circ \rho_{\lambda_j})$ is compact. 
Then $(\ph_j)_j$ satisfies the requirements in the lemma. 
\end{proof} 

\section{Proof of Theorem \ref{151207_1}}

First we prove the following lemma. 

\begin{lem}\label{151205_1} 
Let $\omega$ be any area form on $A := [0, 1] \times S^1$, 
and $U$, $V$ be open neighborhoods of $\{1\} \times S^1$ which are disjoint from $\{0\} \times S^1$. 
For any diffeomorphism 
$\psi: U \to V$ which satisfies $\psi^*\omega = \omega$, 
there exists $\bar{\psi}  \in \Diff(A, \omega)$ which satisfies 
$\bar{\psi}  \equiv \psi$ near $\{1\} \times S^1$
and $\bar{\psi}  \equiv \id$ near $\{0\} \times S^1$. 
\end{lem}
\begin{proof}
Since any orientation-preserving diffeomorphism on $S^1$ is smoothly isotopic to $\id_{S^1}$, 
$\psi$ is smoothly isotopic to $\id$ near $\{1\} \times S^1$. 
Hence there exists $f \in \Diff(A)$ 
such that $f \equiv \psi$ near $\{1\} \times S^1$ and $f \equiv \id$ near $\{0\} \times S^1$. 

$(f^*\omega - \omega)|_{\interior A}$ represents $0$ in 
$H^2_{c,\dR}(\interior A)$. 
Thus there exists $\eta \in \Omega^1(A)$, which vanishes near $\partial A$ and satisfies $d\eta=f^*\omega - \omega$. 
For any $t \in [0,1]$, let $\omega_t:= \omega + t (f^*\omega - \omega)$, 
and define a vector field $X_t$ by $i_{X_t}\omega_t = \eta$. 
Then $X_t \equiv 0$ near $\partial A$. 
Let $(g_t)_{t \in [0,1]}$ be the isotopy on $A$ defined by $g_0 = \id_A$ and $\partial_t g_t = X_t(g_t)$. 
Then, $\bar{\psi} := f \circ g_1$ satisfies the requirements in the lemma. 
\end{proof} 

Let us start the proof of Theorem \ref{151207_1}. 
We may assume that $S$ is connected. 
Let $\ph \in \Ham(S,\omega)$, and take $H \in C^\infty([0,1] \times S)$ such that $\ph = \ph^1_H$. 
By the Arnold conjecture for surfaces (see \cite{Floer_surface} and the references therein), 
there exists $q \in \fix(\ph)$ such that
$(\ph^t_H(q))_{t \in [0,1]}$ forms a contractible loop on $S$. 

\begin{lem}\label{151208_1} 
For any $\beta \in \Omega^1(S \setminus \{q\})$ such that $d\beta=\omega$, 
$\ph^*\beta - \beta$ is exact. 
\end{lem} 
\begin{proof}
It is sufficient to show that any $\gamma: \R/\Z \to S \setminus \{q\}$ satisfies $\int_\gamma \ph^*\beta - \beta =0$.
Let us define $\Gamma: [0,1] \times \R/\Z \to S$ by 
$\Gamma(t,\theta):= \ph^t_H(\gamma(\theta))$, then $\int_\Gamma \omega = 0$. 
It is sufficient to prove the following claim: 
\begin{quote}
\textbf{Claim:} 
There exists a smooth family of maps $(\Gamma_s: [0,1] \times \R/\Z \to S)_{s \in [0,1]}$ such that 
$\Gamma_0=\Gamma$, 
$q \notin \image \Gamma_1$, 
and $\Gamma_s|_{\{0,1\} \times \R/\Z} = \Gamma|_{\{0,1\} \times \R/\Z}$ for any $s \in [0,1]$. 
\end{quote}
Indeed, once we have established the claim, we can complete the proof by 
\[
\int_\gamma \ph^*\beta - \beta = \int_{\Gamma_1} d\beta = \int_{\Gamma_1} \omega = \int_{\Gamma_0} \omega = 0. 
\]

Now let us prove the above claim. 
Since $(\ph^t_H(q))_{t \in [0,1]}$ forms a contractible loop on $S$, there exists a smooth map 
$C: [0,1]^2 \to S$ such that
\[
C(0,t)= \ph^t_H(q), \quad C(1,t) = q, \quad C(s,0) = C(s,1) = q \qquad(\forall s,t \in [0,1]). 
\] 
Then there exists a smooth family of vector fields $(\xi_{s,t})_{(s,t) \in [0,1]^2}$, where $\xi_{s,t}$ is a smooth vector field on $S$ for each $(s,t) \in [0,1]^2$, 
such that 
\[ 
\xi_{s,0}= \xi_{s,1}=0\quad (\forall s \in [0,1]), \quad  
\xi_{s,t}(C(s,t)) = \partial_s C(s,t) \quad(\forall  s, t \in [0,1]). 
\] 
Let $(\Phi_{s,t})_{(s,t) \in [0,1]^2}$ denote the smooth family of isotopies on $S$ defined by 
\[ 
\Phi_{0,t} = \id_S \quad(\forall t \in [0,1]), \quad 
\partial_s \Phi_{s,t} = \xi_{s,t}(\Phi_{s,t}) \quad(\forall s,t \in [0,1]). 
\] 
Then it is easy to see that 
\[
C(s,t) = \Phi_{s,t}(\ph^t_H(q)) \quad( \forall s,t \in [0,1]), \qquad
\Phi_{s,0} = \Phi_{s,1} = \id_S \quad( \forall s \in [0,1]). 
\] 

Now let us define $\Gamma_s$ by $\Gamma_s(t,\theta):= \Phi_{s,t}(\Gamma(t,\theta))$. 
The properties $\Gamma_0 = \Gamma$ and $\Gamma_s|_{\{0,1\} \times \R/\Z} = \Gamma|_{\{0,1\} \times \R/\Z}\, (\forall s \in [0,1])$ are easy to see. 
The property $q \notin \image \Gamma_1$ can be confirmed by 
\[ 
\Gamma_1(t,\theta) = \Phi_{1,t}(\Gamma(t,\theta)) = \Phi_{1,t} \circ \ph^t_H (\gamma(\theta)) \ne \Phi_{1,t} \circ \ph^t_H(q) = C(1,t) = q. 
\] 
The middle inequality follows since $\Phi_{1,t}$ and $\ph^t_H$ are bijections and $q \notin \image \gamma$. 
\end{proof} 

Let $U$ be any nonempty open set in $S$. 
We may assume that $U$ is diffeomorphic to $\R^2$ and 
$\bar{U}$ (the closure of $U$ in $S$) is disjoint from $q$. 
We are going to show that, there exists a sequence $(\ph_j)_j$ in $\Ham(S, \omega)$ which satisfies 
the requirements in Theorem \ref{151207_1}. 

Let us take a local chart $(x,y)$ near $q$ such that $q=(0, 0)$ and $\omega=dx \wedge dy$. 
By adding a $C^\infty$ small perturbation to $\ph$, 
we may assume that the following conditions are satisfied: 
\begin{itemize}
\item The eigenvalues of $d\ph(q)$ are in 
$\{ z \in \C \mid z^3 \ne 1, \, z^4 \ne 1 \}$. 
In particular $q$ is a nondegenerate fixed point of $\ph$. 
\item 
With respect to the local chart $(x,y)$, $\ph$ is real-analytic near $(0,0)$. 
\end{itemize} 
By the first condition, the fixed point $q$ is either 
hyperbolic (i.e., the eigenvalues of $d\ph(q)$ are in $\R \setminus \{ \pm 1\}$) or 
elliptic (i.e., the eigenvalues are in $\{z \in \C \mid |z|=1\} \setminus \{\pm 1\}$). 
We consider the two cases separately. 

\textbf{The case $q$ is hyperbolic} 

According to the result by Moser (\cite{Moser} Theorem 1), 
there exists a local chart $(X,Y)$ defined near $q$ such that
$q=(0,0)$, $\omega=dX \wedge dY$ and 
\[
\ph(X, Y) = (u(XY) X,  u(XY)^{-1}Y), 
\]
where $u(t)$ is a real-analytic function defined near $t=0$. 
Notice that $u(0)$ is an eigenvalue of $d\ph(q)$, in particular nonzero. 

Let us take sufficiently small $\ep>0$, and let 
$U_\ep:= \{(X,Y) \mid X^2+Y^2<\ep\}$. 
Let us define a diffeomorphism 
$F: (0, \ep/2) \times \R/2\pi \Z \to U_\ep \setminus \{q\}$ by
\[
F(r, \theta):= \sqrt{2r} (\cos \theta, \sin \theta), 
\]
then $F^*(dX \wedge dY) = dr \wedge d\theta$, 
and $F^{-1} \circ \ph \circ F$ is defined on $(0, \delta) \times \R/2\pi \Z$ for sufficiently small $\delta>0$.
Let us define $R$ and $\Theta$ by 
\[
(R(r, \theta), \Theta(r, \theta)):= F^{-1} \circ \ph \circ F(r, \theta). 
\]
Setting $v(r, \theta):= u(r \cdot \sin 2\theta)$, 
direct computations show 
\begin{align*} 
R(r, \theta) &= r \cdot (v(r, \theta)^2 \cos^2\theta + v(r, \theta)^{-2} \sin^2\theta), \\
\tan \Theta(r, \theta) &= v(r, \theta)^{-2} \tan \theta.
\end{align*} 
Hence, for sufficiently small $\delta'>0$, 
$F^{-1} \circ \ph \circ F$ uniquely extends to $(-\delta', \delta) \times \R/2\pi\Z$
as a real-analytic map, which we denote by $\psi$. 
Let $\omega_0:=dr \wedge d\theta$.
Since $\psi^*\omega_0 - \omega_0$ is real-analytic and vanishes on $(0, \delta) \times \R/2\pi\Z$, 
it vanishes on $(-\delta', \delta) \times \R/2\pi\Z$, thus $\psi^*\omega_0 = \omega_0$. 
By Lemma \ref{151205_1}, there exists $\bar{\psi}  \in \Diff([-1,0] \times \R/2\pi \Z, \omega_0)$ such that 
$\bar{\psi}  \equiv \psi$ near $\{0\} \times \R/2\pi \Z$, and 
$\bar{\psi}  \equiv \id$ near $\{-1\} \times \R/2\pi \Z$. 

Let $\bar{S}:= ([-1, \delta) \times \R/2\pi \Z) \cup_F (S \setminus \{q\})$. 
We define an area form $\bar{\omega}$ on $\bar{S}$ by 
\[
\bar{\omega}|_{[-1, \delta) \times \R/2\pi \Z} = \omega_0, \qquad 
\bar{\omega}|_{S \setminus \{q\}} = \omega.
\]
We also define $\bar{\ph} \in \Diff(\bar{S}, \bar{\omega})$ by
$\bar{\ph}|_{[-1,0] \times \R/2\pi \Z} = \bar{\psi} $ and 
$\bar{\ph}|_{S \setminus \{q\}} = \ph$. 

By Lemma \ref{151208_1}, one can apply 
Lemma \ref{151207_2} for 
$(\bar{S}, \bar{\omega}, \bar{\ph})$. 
Then there exists a sequence $(\bar{\ph}_j)_j$ in $\Diff(\bar{S}, \bar{\omega})$ such that 
$\supp (\bar{\ph}^{-1} \circ \bar{\ph}_j) \subset U$, 
$\mca{P}(\bar{\ph}_j) \cap U \ne \emptyset$ for every $j$, 
and $\lim_{j \to \infty} \bar{\ph}_j = \bar{\ph}$. 
For each $j$, $\bar{\ph}_j|_{S \setminus \{q\}}$ extends to 
$\ph_j \in \Diff(S,\omega)$ by $\ph_j(q) := q$. 
Moreover $\ph_j \in \Ham(S, \omega)$, since 
$\supp (\ph^{-1} \circ \ph_j) \subset U$ and $U$ is diffeomorphic to $\R^2$.
Then, the sequence $(\ph_j)_j$ satisfies the requirements in Theorem \ref{151207_1}. 

\textbf{The case $q$ is elliptic} 

We assumed that the eigenvalues of $d\ph(q)$ are in $\{z \in \C \mid z^3 \ne 1, \, z^4 \ne 1 \}$. 
Then, there exists a local chart $(X,Y)$ near $q$ such that 
$q=(0,0)$, $\omega = dX \wedge dY$ and
\[
\ph(X,Y) = (\cos \theta(X,Y) X  - \sin \theta(X,Y) Y,  \sin \theta(X, Y) X + \cos \theta(X,Y) Y) + O_4(X,Y),
\]
where $\theta(X,Y)= \theta_0 + \theta_1 (X^2 + Y^2)$ ($\theta_0, \theta_1$ are real constants), 
and 
$O_4$ is a real-analytic map whose expansion involves terms of order $\ge 4$ only 
(see \cite{Siegel_Moser} Section 32 and Section 23, pp.172--173). 

By adding a $C^\infty$ small perturbation to $\ph$, we may assume that $\theta_1 \ne 0$. 
Then, there exists a neighborhood $D$ of $q$ which is diffeomorphic to $D^2$, 
preserved by $\ph$ and sufficiently close to $q$ such that $D \cap \bar{U} = \emptyset$ 
(see \cite{Siegel_Moser} Section 34). 
$\partial D$ is a so called KAM invariant circle. 

Again by Lemmas \ref{151205_1} and \ref{151208_1}, 
one can apply Lemma \ref{151207_2} to conclude that 
there exists a sequence $(\ph'_j)_j$ in $\Diff(S \setminus D, \omega)$ such that 
$\mca{P}(\ph'_j) \cap U \ne \emptyset$, 
$\supp (\ph^{-1} \circ \ph'_j) \subset U$ for every $j$, and 
$\lim_{j \to \infty} \ph'_j = \ph|_{S \setminus D}$. 
Every $\ph'_j$ extends to $\ph_j \in \Diff(S, \omega)$
by setting $\ph_j| _D:= \ph|_D$. 
$\ph_j \in \Ham(S, \omega)$ 
since $\supp (\ph^{-1} \circ \ph_j) \subset U$ and $U$ is diffeomorphic to $\R^2$.
The sequence $(\ph_j)_j$ satisfies the requirements in Theorem \ref{151207_1}. 
\qed

\textbf{Acknowledgements.} 
We are grateful to the referees for their comments which are very helpful to improve presentations of the paper. 
We also appreciate Kenji Fukaya, Viktor Ginzburg, Yi-Jen Lee and Kaoru Ono for useful conversations. 
MA is supported by JSPS Grant-in-Aid for Scientific Research (C) (26400085). 
KI is supported by JSPS Grant-in-Aid for Young Scientists (B) (25800041), 
he also acknowledges the Simons Center for Geometry and Physics at Stony Brook University for the great working environment. 


\begin{thebibliography}{ZZZ}

\bibitem{Anosov_Zhuzhoma} 
Anosov, D. V., 
Zhuzhoma, E. V.:
Closing lemmas. 
Differ. Equ. \textbf{48} (13), 
1653--1699 (2012) 

\bibitem{CGH}
Cristofaro-Gardiner, D., 
Hutchings, M.:
From one Reeb orbit to two. 
J. Differential Geom. \textbf{102} (1), 
25--36 (2016) 

\bibitem{CGHR}
Cristofaro-Gardiner, D., 
Hutchings, M., 
Ramos, V.G.B.:
The asymptotics of ECH capacities. 
Invent. Math. \textbf{199} (1), 
187--214 (2015) 

\bibitem{Floer_surface}
Floer, A.: 
Proof of the Arnold conjecture for surfaces and generalizations to certain K\"{a}hler manifolds.
Duke Math. J. \textbf{53} (1), 
1--32 (1986) 

\bibitem{FLC}
Franks, J., 
Le Calvez, P.:
Regions of instability for non-twist maps. 
Ergodic Theory Dynam. Systems. \textbf{23} (1), 
111--141 (2003) 

\bibitem{HWZ_3D}
Hofer, H., 
Wysocki, K., 
Zehnder, E.:
The dynamics on three-dimensional strictly convex energy surfaces. 
Ann. of Math. (2) 
\textbf{148} (1), 
197--289 (1998)

\bibitem{Hutchings_Quantitative} 
Hutchings, M.: 
Quantitative embedded contact homology. 
J. Differential Geom. \textbf{88} (2), 
231--266 (2011)

\bibitem{Hutchings_Calabi}
Hutchings, M.: 
Mean action and the Calabi invariant. 
arXiv: 1509.02183v3

\bibitem{Irie}
Irie, K.: 
Dense existence of periodic Reeb orbits and ECH spectral invariants. 
J. Mod. Dyn. \textbf{9}, 
357--363 (2015) 

\bibitem{Moser}
Moser, J.: 
The analytic invariants of an area-preserving mapping near a hyperbolic fixed point. 
Comm. Pure Appl. Math. \textbf{9}, 
673--692 (1956) 

\bibitem{Pugh1}
Pugh, C.C.:
The closing lemma. 
Amer. J. Math. \textbf{89}, 
956--1009 (1967) 

\bibitem{Pugh2}
Pugh, C.C.: 
An improved closing lemma and a General Density Theorem. 
Amer. J. Math. \textbf{89}, 
1010--1021 (1967)

\bibitem{PR}
Pugh, C.C., 
Robinson, C.:
The $C^1$ closing lemma, including Hamiltonians.
Ergodic Theory Dynam. Systems. \textbf{3} (2), 
261--313 (1983) 

\bibitem{Siegel_Moser}
Siegel, C.L., 
Moser, J.K.:
Lectures on Celestial Mechanics.
Classics in Mathematics.
Springer-Verlag, Berlin (1995) 

\bibitem{Smale}
Smale, S: 
Mathematical problems for the next century. 
Math. Intelligencer \textbf{20} (2), 
7--15 (1998) 

\bibitem{Xia}
Xia, Z.: 
Area-preserving surface diffeomorphisms.
Comm. Math. Phys. \textbf{263} (3), 
723--735 (2006) 

\bibitem{Zehnder} 
Zehnder, E.: 
Note on smoothing symplectic and volume-preserving diffeomorphisms. 
pp. 828--854. Lecture Notes in Math., Vol. 597, Springer, Berlin (1977) 

\end{thebibliography}
\end{document}